\documentclass[english]{amsart}

\usepackage{graphicx}
\usepackage{amssymb}
 \usepackage{amsthm}
 \usepackage{amsmath}
 \usepackage{float}

\newtheorem{thm}{Theorem}
\newtheorem{lm}{Lemma}

\begin{document}


\title{The Runge-Kutta Method in Geometric Multiplicative Calculus}


\author{Mustafa Riza}
\email{mustafa.riza@emu.edu.tr}
\address{Department of Physics, Eastern Mediterranean University, Gazimagusa, North Cyprus, via Mersin 10, Turkey}
\author{Hatice Akt\"ore}
\email{hatice.aktore@emu.edu.tr}
\address{Department of Mathematics, Eastern Mediterranean University, Gazimagusa, North Cyprus, via Mersin 10, Turkey}

\begin{abstract}
This paper illuminates the derivation, the applicability, and the efficiency of the Multiplicative Runge-Kutta Method, derived in the framework of geometric multiplicative calculus. The removal of the restrictions of geometric multiplicative calculus to positive-valued functions of real variable and the fact that the multiplicative derivative does not exist at the roots of the function, is presented explicitly to ensure that the proposed method is universally applicable. The error analysis  is also carried out in the framework of geometric multiplicative calculus explicitly. The presented method is applied to various problems and the results are compared to the ones obtained from the Ordinary Runge-Kutta Method. Moreover, for one example, a comparison of the computation time vs. relative error, is worked out, to  illustrate  the general advantage of the proposed method.

\end{abstract}

\keywords{Multiplicative Calculus, Complex Multiplicative Calculus, Runge-Kutta, Initial Value Problems, Differential Equations, Numerical Approximation}

\subjclass[2010]{65L06,34K28}

\maketitle
\section{Introduction}
\label{sec:intro}

The invention of Multiplicative Calculus can be dated back to 1972, when Micheal Grossman and Robert Katz finished their book on Non-Newtonian Calculi \cite{GK}, where they proposed nine different Non-Newtonian Calculi. Later Micheal Grossmann elaborated the bigeometric multiplicative calculus in \cite{GR}. Bigeometric multiplicative calculus was also proposed independently by Cordova-Lepe \cite{C1} under the name proportional calculus. Although, Volterra and Hostinsky proposed a kind of flavor of multiplicative calculus in \cite{V}, we can not date the invention of multiplicative calculus back to 1938. After 25 years of silence in this field Bashirov et al presented a mathematical precise description of the geometric multiplicative calculus in  \cite{BKO}.
This work initiated  numerous studies in the field of multiplicative calculus. Several multiplicative numerical approximation methods have been proposed  and discussed like \cite{ozyapici2007exponential,ozyapici2008notes,ROK,MG, ozyapici2013multiplicative,mrbe2014}. Moreover, multiplicative calculus has  found its way into biomedical image analysis \cite{florack12} and modelling with differential equations  \cite{BMTO}. Furthermore,  the Runge-Kutta method was developed in the framework of bigeometric calulus for applications in dynamic systems by Aniszewska et al. in \cite{A}. A more exact bigeometric Runge-Kutta method was proposed by Riza and Eminaga \cite{mrbe2014}, based on the bigeometric taylor  theorem derived in  \cite{mrbe2014}.

One drawback of  Multiplicative Calculus, generally put forward, is that multiplicative calculus can only be applied to positive-valued functions of real variable. This restriction can be circumvented by using complex multiplicative calculus. The first attempt was presented by Uzer in \cite{U}. A mathematically precise description of the complex geometric multiplicative calculus was given by Bashirov and Riza  in  \cite{BR,bashirov2011complex}. The fact, that the derivative is a local property suggests the extension to the complex domain. So, a simple change to complex-valued functions of real variable removes the restriction to deal with purely positive functions and allows us to treat the real and imaginary part independently as the multiplicative Cauchy-Riemann conditions become trivial in this case. 

In section \ref{sec:mrk}, the multiplicative Euler method and the 4th order multiplicative Runge-Kutta (MRK4) method for positive-valued functions of real variable will be elaborated, and the extension to complex-valued functions of real variable will be presented.  Another wellknown drawback of multiplicative calculus is the breakdown of the multiplicative derivative at the roots of the functions. Section \ref{sec:extension} covers also the the solution to that problem. The error analysis for the geometric multiplicative Runge-Kutta method presented in section \ref{sec:error}, carried out in analogy to the ordinary Runge-Kutta method like e.g. in \cite{stbul93}, shows that the error becomes considerably smaller for  the same step size compared to the ordinary Runge-Kutta method.  In section \ref{sec:ex},  the geometric multiplicative Runge-Kutta Method will be applied to a multiplicative initial value problem, not involving the exponential or logarithmic functions, with a known closed form solution. The results of the application of the multiplicative Runge-Kutta method are compared to the results of the ordinary Runge-Kutta method for a fixed step width $h$. Furthermore, the computation time versus the relative error with varying step width for this example are presented to show the superiority of the proposed method. Based on the Baranyi model for bacterial growth \cite{baranyi95,baranyi94} using differential equations, the multiplicative Runge-Kutta method was applied on the bacterial growth in food modelled by Huang \cite{huang12,huang10,huang08}, and compared to the results from the ordinary Runge-Kutta method. As an example for a coupled system of multiplicative initial value problems, a second order differential equation, with well-known closed form solution also used in \cite{ROK}, is used to compare the multiplicative Runge-Kutta method with the multiplicative finite difference method. All examples show the superiority of the multiplicative Runge-Kutta method, with respect to error as well as performance. 
Finally, all  findings will be summarised in the section \ref{sec:con}.

  In order to ease the reading of this paper, we will use multiplicative calculus synonymical to geometric multiplicative calculus.

\section{Multiplicative Runge Kutta Method for real-valued functions of real variable}
\label{sec:mrk}

In this section, the Multiplicative Runge-Kutta Method, also referred as MRK-method in the following, will be derived for the 2nd order case exemplarily explicitly. Only the starting equations and the results of the  fourth order MRK-method will be presented. 

The methods being derived in the following will be used to find suitable approximations to the solution of multiplicative initial value problems of the form:
\begin{equation}
y^*(x) = f(x,y),
\label{eq:mdeq}
\end{equation}
with the initial condition
\begin{equation}
y(x_0) =y_0.
\label{eq:miv}
\end{equation}

\subsection{2nd order MRK method}

The simplest approach to find an approximation to the solution of the differential equation \eqref{eq:mdeq} with the initial value \eqref{eq:miv} is the second order Runge-Kutta Method, also known as the Euler method. In analogy to the ordinary Euler method, we will derive in the following the second order Multiplicative Runge-Kutta  method or the Multiplicative Euler method by making the ansatz:

\begin{equation}
y(x+h)=y(x)\cdot f_0^{ah}\cdot f_1^{bh},
\label{eq:meu}
\end{equation}
where
\begin{eqnarray}
f_0&=&f(x,y), \text{ and } \label{eq:rkf0}\\
f_1&=&f(x+ph,y\cdot f_0^{qh}). \label{eq:rkf1}
\end{eqnarray}

The multiplicative taylor expansion of $y(x+h)$ up to order 2 is given as

\begin{equation}
y(x+h)=y(x)\cdot y^*(x)^h\cdot y^{**}(x)^{h^2/2}\cdot... \qquad .
\end{equation}

Remembering that, 

\begin{equation}
y^*(x)=f(x,y) \quad \text{and} \quad y^{**}(x)=f_x^*(x,y)\cdot f_y^*(x,y)^{y\ln f(x,y)}
\end{equation}

the multiplicative Taylor expansion of $y(x+h)$ becomes

\begin{equation}
y(x+h)=y(x)\cdot f(x,y)^h\cdot f^*_x(x,y)^{h^2/2}\cdot f^*_y(x,y)^{y\ln f(x,y)h^2/2},
\label{eq:mtaylor}
\end{equation}

where $f^*_x(x,y)$ denotes the multiplicative partial derivative with respect to $x$ and $f^*_y(x,y)$  with respect to $y$ respectively. 

In order to compare \eqref{eq:mtaylor} with \eqref{eq:meu} we need to expand also $f_1$ using the multiplicative taylor theorem up to order 1 as the power of the ansatz \eqref{eq:meu} also includes one $h$. Recalling that  $y$ is a function of $x$ the taylor expansion for $f_1$ becomes by the application of the multiplicative chain rule 
\begin{equation*}
f_1=f(x,y)\cdot f^*_x(x,y)^{ph}\cdot f^*_y(x,y)^{yqh\ln f_0}. 
\end{equation*}

With  $f_0=f(x,y)$,  the Taylor expansion of $f_1$ up to order 1 in $h$ becomes

\begin{equation}
f_1=f(x,y)\cdot f^*_x(x,y)^{ph}\cdot f^*_y(x,y)^{yqh\ln f(x,y)}.\label{eq:f1exp}
\end{equation}

Then, by substituting \eqref{eq:f1exp} and \eqref{eq:rkf0}  in \eqref{eq:meu}, we get the Multiplicative Runge-Kutta expansion for the comparison with the multiplicative Taylor expansion of \eqref{eq:mtaylor} as 

\begin{equation}
y(x+h)=y(x)\cdot f(x,y)^{(a+b)h}\cdot f_x(x,y)^{bph^2}\cdot f_y(x,y)^{y\ln f(x,y)bqh^2}.
\label{eq:meuexp}
\end{equation}

Comparison of the powers of $f(x,y)$ and its partial derivatives in  \eqref{eq:meuexp} with \eqref{eq:mtaylor} up to order 2 in $h$ gives:

\begin{eqnarray}
a+b&=&1, \label{eq:mr1}\\
bp&=&\frac{1}{2},\label{eq:mr2}\\
bq&=&\frac{1}{2}. \label{eq:mr3}.
\end{eqnarray}

Obviously, we have infinitely many solutions of the equations \eqref{eq:mr1}-\eqref{eq:mr3}, as the number of unknowns is greater than the number of equations. Furthermore, we can see that $p=q$ and $a+b=1$, which can be easily reflected in analogy to the regular Butcher Tableau \cite{butcher75} as the multiplicative Butcher Tableau.

One possible choice of the parameters $a,b,p,$ and $q$ is as following:
 \[
 a=\frac{1}{2}, b=\frac{1}{2}, p=1, \text{ and } q=1.
 \]  
 
 Here, we can see that we evaluate the function at the endpoints of the interval, and give equal weights to the contributions of $f_0$ and $f_1$, resulting in the multiplicative Euler method formulae 

\begin{eqnarray*}
y(x+h)&=&y(x)\cdot f_0^{\frac{h}{2}}\cdot f_1^{\frac{h}{2}},\\
f_0&=&f(x,y), \quad \text{and}\\
f_1&=&f(x+h,y\cdot f_0^{h}).
\end{eqnarray*}

Of course, depending on the problem the parameters can be also chosen differently to optimise the solutions, satisfying the equations \eqref{eq:mr1}-\eqref{eq:mr3}.

\subsection{4th order MRK method}

In practice, mainly the 4th order Runge-Kutta method is used. In analogy to the above described 2nd  order multiplicative Runge-Kutta method, we will now employ the 4th order multiplicative Runge-Kutta method. Consequently, we make the following ansatz

\begin{eqnarray*}
y(x+h)&=&y(x)\cdot f_0^{ah}\cdot f_1^{bh}\cdot f_2^{ch}\cdot f_3^{dh},\\
f_0&=&f(x,y),\\
f_1&=&f(x+ph,y\cdot f_0^{qh}),\\
f_2&=&f(x+p_1h,y\cdot f_0^{q_1h}\cdot f_1^{q_2h}),\\
f_3&=&f(x+p_2h,y\cdot f_0^{q_3h}\cdot f_1^{q_4h}\cdot f_2^{q_5h}).
\end{eqnarray*}

Again we need to find the Taylor expansions of $f_0 , f_1 , f_2$ and $f_3$ in order to be substituded into the 4th order multiplicative Runge-Kutta formula, and compare it with the Taylor expansion of $y(x+h)$  up to order 4. After a lengthy calculation, we get by comparison, the following set of equations 

\begin{eqnarray}
p&=&q, \label{m4p}\\
p_1&=& q_1+q_2, \label{m4p1} \\
p_2 &=& q_3+q_4+q_5, \label{m4p2}
\end{eqnarray}

and

\begin{eqnarray}
a+b+c+d&=&1, \label{m4e13}\\
bp+cp_1+dp_2&=&\frac{1}{2},  \label{m4e23}\\
bp^2+cp_1^2+dp_2^2&=&\frac{1}{3}.  \label{m4e33}
\end{eqnarray} 

Resulting in the multiplicative Butcher Tableau
\[
\begin{array}{c|cccc}
0& & & &\\
p& q& & &\\
p_1& q_1& q_2& &\\
p_2& q_3& q_4& q_5&\\
\hline
& a & b& c & d 
\end{array}
\]

We can easily see if the function $f(x,y)$ is independent of $y$, the result is independent from the selection of $q_1$ to $q_5$,  and therefore any selection will give the same result.

\subsection{Extension to complex valued functions of real variable}
\label{sec:extension}

One of the drawbacks of Multiplicative Calculus, generally put forward, is its restriction to positive valued functions of real variable. In order to overcome this restriction the theory of Multiplicative Calculus was extended to the complex domain. It is well-known from complex analysis that  the differentiation rules are a little bit more complicated for complex valued functions of complex variable as the Cauchy-Riemann conditions have to be satisfied. But, here we are only interested in complex valued functions of real variable, which simplifies the issue drastically, as the multiplicative counterparts of the Cauchy-Riemann conditions have not to be taken into account and the differentiation can be carried out independently for the real and the imaginary part. As illustrated in \cite{BR} the multiplicative derivative can be calculated everywhere except at the point $0+0i$ in the complex plain. So the 4th order multiplicative Runge-Kutta Method can be extended to negative valued functions as the phase factor is responsible for the change of the sign. The only problem that  could not be solved by extending Multiplicative Calculus to the complex domain is that the Multiplicative derivative is not defined at the roots of the function. So, a switch to Newtonian Calculus becomes inevitable at these points. In every step of the Multiplicative Runge-Kutta Method, we get the value of the function at this point and its multiplicative derivative at this point and use the ordinary Runge-Kutta Method for a couple of steps until the multiplicative derivative becomes again reasonably large and this values are then used as input of the Multiplicative Runge-Kutta method. The results are reasonably good, and often even better than using the ordinary Runge-Kutta method alone. 

If we assume that $f(x_{i-1})>0$, and $f(x_{i+1})<0$ and that the function is decreasing, then accordingly there must be a point $\xi \in [x_{i-1}, x_{i+1}]$ where $f(\xi)=0$. In this case the multiplicative derivative of $f(x)$ is not defined at $\xi$. Therefore, the Multiplicative Runge-Kutta method will be applied on the interval $[x_0, x_{i-1}]$, and on the interval $[x_{i+1},x_n]$. On the interval $[x_{i-1},x_{i+1}]$ we apply the ordinary Runge-Kutta Method, using the values $f(x_{i-1})$ and $f^*(x_{i-1})$ calculated by the Multiplicative Runge-Kutta Method as input for the ordinary Runge-Kutta Method, and vice versa for the point $x_{i+1}$.

\begin{figure}[H]
\centerline{\includegraphics[width=5cm]{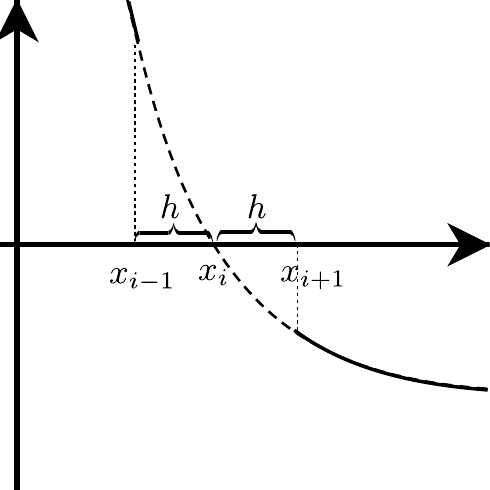}}
\caption{Bypass the roots where the multiplicative derivative becomes undefined. The dashed line denotes the region where the ordinary Runge-Kutta method is applied to prevent the multiplicative derivate to become infinite. The multiplicative Runge-Kutta method is applied in the region of the solid line.}
\end{figure}

The handover has been tested on several examples,working properly.


\section{Error Analysis}
\label{sec:error}
\subsection{Convergence Of One-Step Methods}
In this section we  examine the convergence behaviour of our one step method as $h\rightarrow0$ of an approximate solution $\eta(x;h)$. We assume that $f$ is one time $*$-differentiable on the interval $(a,b)$ and $y(x)$ denotes the exact solution of the initial-value problem

\begin{equation*}
y^*=f(x,y), \quad\quad y(x_0)=y_0.
\end{equation*} 

Let $\Phi(x,y;h)$ define a one-step method,

\begin{eqnarray*}
\eta_0&:=&y_0,\\
\quad \text{for} \quad i=0,1,\ldots :\\
\eta_{i+1}&:=&\eta_i\Phi(x_i,\eta_i;h)^h,\\
x_{i+1}&:=&x_i+h,
\end{eqnarray*}

which for $x\in R_h:=\{x_0+ih \, | \,i=0,1,2,\ldots\}$ produces the approximate solution $\eta(x;h)$:

\begin{equation*}
\eta(x;h):=\eta_i, \quad\quad \text{if} \quad x=x_0+i h.
\end{equation*}

Let  $x$ and $y$ be arbitrary, but fixed, and let $z(t)$ be the exact solution of the initial-value problem

\begin{equation}
z^*(t)=f(t,z(t)), \quad\quad z(x)=y,
\label{eq:z}
\end{equation}

with initial values $x$, $y$. Then the function

\begin{equation*}
\Delta(x,y;h):=
\begin{cases}
    \left(\frac{z(x+h)}{y}\right)^h  & \text{if } h\neq0,\\
     f(x,y)     & \text{if }  h=0 
\end{cases}
\end{equation*}

represents the multiplicative ratio function of the exact solution $z(t)$ of \eqref{eq:z} for step size $h$, while $\Phi(x,y;h)$ is the multiplicative ratio function for step size $h$ of the approximate solution of \eqref{eq:z} produced by $\Phi$. The  multiplicative ratio function is the multiplicative counterpart to the difference quotient in Newtonian calculus. 

The magnitude of the ratio

\begin{equation*}
\tau(x,y;h):=\frac{\Delta(x,y;h)}{\Phi(x,y;h)}
\end{equation*}

indicates how well the value $z(x+h)$ at $x+h$ obeys the equation of the one-step method.

One calls $\tau(x,y;h)$ the {\em{multiplicative  local discretization  error}} at the point $(x,y)$. For a reasonable one-step method one will require that

\begin{equation*}
\lim_{h\to0}\tau(x,y;h)=1
\end{equation*}

We are interested in the behaviour of the {\em{multiplicative global  discretization  error}}

\begin{equation*}
e(x;h):=\frac{\eta(x;h)}{y(x)}
\end{equation*}

for fixed $x$ and $h\rightarrow0$, $h\in H_x:=\left\{\frac{(x-x_0)}{n} \, | \, n=1,2,\ldots\right\}$. Since $e(x;h)$, like $\eta(x;h)$, is only defined for $h\in H_x$, we have to study the convergence of  

\begin{equation*}
e(x;h_n), \quad\quad h_n:=\frac{x-x_0}{n}, \quad\quad \text{as} \quad n\rightarrow\infty.
\end{equation*}

We say that the one-step method is $convergent$ if 
\begin{equation*}
\lim_{n\to \infty}e(x;h_n)=1
\end{equation*}

for all $x\in[a,b]$ and all functions $f$ being one time $*$-differentiable on the interval $(a,b)$.\\
For $f$ being $p$-times $*$-differentiable on $(a,b)$, methods of order $p>0$ are convergent, and  satisfy

\begin{equation*}
e(x;h_n)=O(e^{h_n^p}).
\end{equation*}

The order of the multiplicative global discretization error is thus equal to the order of the multiplicative local discretization error.  

\begin{lm}\label{eq:thm}

If the numbers $\xi_i$ satisfy estimates of the form 

\begin{equation*}
|\xi_{i+1}|\leq|\xi_i|^{(1+\delta)}B, \quad\quad  \delta>0, \quad\quad B\geq0, \quad\quad i=0,1,2,...,  
\end{equation*}

then 

\begin{equation*}
|\xi_n|\leq|\xi_0|^{e^{n\delta}}B^{\frac{e^{n\delta}-1 }{\delta}}
\end{equation*}

\end{lm}

\begin{proof}

From the assumptions we get immediately

\begin{eqnarray*}
|\xi_1| &\leq& |\xi_0|^{(1+\delta)}B \\
|\xi_2| &\leq& |\xi_0|^{(1+\delta)^2}B^{1+(1+\delta)}\\
\vdots\\
|\xi_n| &\leq& |\xi_0|^{(1+\delta)^n}B^{[1+(1+\delta)+(1+\delta)^2+...+(1+\delta)^{n-1}]}\\
&=&|\xi_0|^{(1+\delta)^n}B^{\frac{(1+\delta)^n-1}{\delta}}\\
&\leq&|\xi_0|^{e^{n\delta}}B^{\frac{e^{n\delta}-1 }{\delta}}
\end{eqnarray*}

since $0<1+\delta\leq e^{\delta}$ for $\delta>-1$.

\end{proof}

\begin{thm}

Consider, for $x_0 \in [a,b]$, $y_0 \in \mathbb{R}$, the initial-value problem 

\begin{equation*}
y^*=f(x,y), \quad y(x_0)=y_0,
\end{equation*}

having the exact solution $y(x)$. Let the function $\Phi$ be continuous on 

\begin{equation*}
G:=\left\{(x,y,h)\, |\, a \leq x \leq b,\, \left|\frac{y}{y(x)}\right|\leq \gamma, \, 0 \leq |h| \leq h_0 \right\}, \, h_0>0, \, \gamma>1,
\end{equation*}

and there exist positive constants $M$ and $N$ such that 

\begin{equation*}
\left|\frac{\Phi(x,y_1;h)}{\Phi(x,y_2;h)}\right| \leq \left|\frac{y_1}{y_2}\right|^M
\end{equation*}

for all $(x,y_i,h) \in G$, $i=1,2,$ and 

\begin{equation*}
\left|\tau(x,y(x);h)\right| = \left|\frac{\Delta(x,y(x);h)}{\Phi(x,y(x);h)}\right| \leq e^{N|h|^p}, \quad p>0
\end{equation*}

for all $x\in[a,b],\, |h|\leq h_0$. Then there exists an $\overline{h}$, $0<\overline{h}\leq h_0$, such that for the multiplicative global discretization error $e(x;h)=\frac{\eta(x;h)}{y(x)}$,

\begin{equation*}
|e(x;h_n)| \leq e^{|h_n|^pN\frac{e^{M|x-x_0|}-1}{M}}
\end{equation*}

for all $x\in[a,b]$ and all $h_n=\frac{x-x_0}{n}$, $n=1,2,\dots$, with $|h_n|\leq \overline{h}$. If $\gamma=\infty$, then $\overline{h}=h_0$.

\end{thm}

\begin{proof}

The function

\begin{equation*}
\widetilde{\Phi}(x,y;h)=
\begin{cases}
    \Phi(x,y;h)  & \text{if } (x,y;h) \in G\\
    \Phi(x,y(x)\gamma;h)     & \text{if }  x\in[a,b], \, |h|\leq h_0, \, y\geq y(x)\gamma \\
    \Phi(x,\frac{y(x)}{\gamma};h)      & \text{if }  x\in[a,b], \, |h|\leq h_0, \, y\leq \frac{y(x)}{\gamma} 
\end{cases}
\end{equation*}

is evidently continuous on $\widetilde{G}:=\{(x,y,h)\, | \, x\in[a,b], \, y\,\in\mathbb{R}, \, |h|\geq h_0 \}$ and satisfies the condition

\begin{equation}
\left|\frac{\widetilde{\Phi}(x,y_1;h)}{\widetilde{\Phi}(x,y_2;h)}\right| \leq \left|\frac{y_1}{y_2}\right|^M
\label{eq:M}
\end{equation}

for all $(x,y_i,h)\in\widetilde{G}, i=1,2$, and  because of $\widetilde{\Phi}(x,y(x);h)=\Phi(x,y(x);h)$, also the condition

\begin{equation}
\left|\frac{\Delta(x,y(x);h)}{\widetilde{\Phi}(x,y(x);h)}\right| \leq e^{N|h|^p}, \quad \text{for}\,\, x\in[a,b], \, |h|\leq h_0.
\label{eq:N}
\end{equation}
is satisfied.Let the one-step method generated by $\widetilde{\Phi}$ furnish the approximate values $\widetilde{\eta_i}:=\widetilde{\eta}(x_i;h)$ for $y_i:=y(x_i)$, $x_i:=x_0+ih$:

\begin{equation*}
\widetilde{\eta}_{i+1}=\widetilde{\eta_i} \cdot \widetilde{\Phi}(x_i,\widetilde{\eta_i};h)^h.
\end{equation*}

In view of

\begin{equation*}
y_{i+1}=y_i \cdot \Delta(x_i,y_i;h)^h,
\end{equation*}

one obtains for the error $\widetilde{e_i}:=\frac{\widetilde{\eta_i}}{y_i}$, the recurrence formula

\begin{equation}
\widetilde{e}_{i+1}=\widetilde{e}_i \cdot \left[\frac{\widetilde{\Phi}(x_i,\widetilde{\eta}_i;h)}{\widetilde{\Phi}(x_i,y_i;h)}\right]^h \cdot \left[\frac{\widetilde{\Phi}(x_i,y_i;h)}{\Delta(x_i,y_i;h)}\right]^h 
\label{eq:e}
\end{equation}

Now from \eqref{eq:M}, \eqref{eq:N} it follows that

\begin{eqnarray*}
\left|\frac{\widetilde{\Phi}(x_i,\widetilde{\eta}_i;h)}{\widetilde{\Phi}(x_i,y_i;h)}\right| &\leq& \left|\frac{\widetilde{\eta}_i}{y_i}\right|^M = |\widetilde{e}_i|^M,\\
\left|\frac{\widetilde{\Phi}(x_i,y_i;h)}{\Delta(x_i,y_i;h)}\right| &\leq& e^{N|h|^p}, 
\end{eqnarray*}

and hence from \eqref{eq:e} we get the recursive estimate 

\begin{equation*}
|\widetilde{e}_{i+1}| \leq |\widetilde{e}_i|^{(1+|h|M)}e^{N|h|^{p+1}}.
\end{equation*}

As we are dealing with an initial value problem, the initial values have to be considered exact, and therefore $\widetilde{e}_0=\frac{\widetilde{\eta}_0}{y_0}=1$, resulting in

\begin{equation}
|\widetilde{e}_k| \leq e^{N|h|^p\frac{e^{k|h|M}-1}{M}}.
\label{eq:ef}
\end{equation}

Now let $x\in[a,b]$, $x\neq x_0$, be fixed and $h:=h_n=\frac{(x-x_0)}{n}$, $n>0$ an integer. Then $x_n=x_0+nh=x$ and from \eqref{eq:ef} with $k=n$, since $\widetilde{e}(x;h_n)=\widetilde{e}_n$ it follows at once that 

\begin{equation*}
|\widetilde{e}(x;h_n)| \leq e^{N|h_n|^p\frac{e^{M|x-x_0|}-1}{M}}
\end{equation*}

for all $x\in[a,b]$ and $h_n$ with $|h_n|\leq h_0$. Since $|x-x_0|\leq|b-a|$ and $\gamma>0$, there exists an $\overline{h}$, $0<\overline{h}\leq h_0$, such that $|\widetilde{e}(x,;h_n)|\leq \gamma$ for all $x\in[a,b]$, $|h_n|\leq \overline{h}$, i.e., for the one-step method generated by $\Phi$, 
\begin{eqnarray*}
\eta_0&=&y_0, \\
\eta_{i+1}&=&\eta_i\Phi(x_i,\eta_i;h),
\end{eqnarray*}

we have for $|h|\leq \overline{h}$, according to the definition of $\widetilde{\Phi}$,

\begin{equation*}
\widetilde{\eta}_i=\eta_i, \quad\quad \widetilde{e}_i=e_i, \quad\quad and \quad\quad \widetilde{\Phi}(x_i,\widetilde{\eta}_i;h)=\Phi(x_i,\eta_i;h).
\end{equation*}

The assertion of the theorem,

\begin{equation*}
|\widetilde{e}(x;h_n)| \leq e^{N|h_n|^p\frac{e^{M|x-x_0|}-1}{M}}
\end{equation*}

thus follows for all $x\in[a,b]$ and all $h_n=\frac{(x-x_0)}{n}$, $n=1,2,\dots,$ with $|h_n|\leq\overline{h}$.

\end{proof}

\section{Examples for the Multiplicative Runge Kutta Method}
\label{sec:ex}

\subsection{Solution of first order multiplicative differential equations}
\label{subsec:exs}	
 
\subsubsection{Square Root Example}
\label{subsec:exsqrt}

As the first example we want to discuss the following multiplicative initial value problem, where no exponential function or logarithm is involved.

\begin{equation}
y^*(x) = e^{\frac{1}{2y^2}}, \qquad y(0)=1,
\label{eq:msqrt}
\end{equation}

where the corresponding Newtonian initial value problem becomes

\begin{equation}
y'(x) = \frac{1}{2y}, \qquad y(0)=1.
\label{eq:nsqrt}
\end{equation}

The general solution of these two initial value problems  \eqref{eq:msqrt} and \eqref{eq:nsqrt} is

\begin{equation}
y(x) =\sqrt{x+1}.
\label{eq:sqrt}
\end{equation}

Application of the 4th order MRK - method and the 4th order RK - method gives the results summarized in the following table.

\begin{table}[H]
\[
\begin{array}{l l l l l l}
x &y_{exact}  &y_{MRK}   & \text{relative}& y_{NEWT} & \text{relative} \\
  &                 &              & err_{MRK} \text{in \% }    &              &  err_{NEWT} \text{in \% } \\
\hline
0   & 1         & 1         & 0              & 1         & 0\\
0.6 & 1.2649111 & 1.2649153 & 3.38\times 10^{-6} & 1.2382302 & 0.021093074\\
1.2 & 1.4832397 & 1.483244  & 2.88\times 10^{-6} & 1.4409643 & 0.028502049\\
1.8 & 1.6733201 & 1.673324  & 2.36\times 10^{-6} & 1.6205072 & 0.031561693\\
2.4 & 1.8439089 & 1.8439125 & 1.97\times 10^{-6} & 1.783364  & 0.032835088\\
3   & 2         & 2.0000034 & 1.69\times 10^{-6} & 1.9334697 & 0.033265139 \\
\hline
\end{array}
\]
\caption{Comparison of the Multiplicative Runge-Kutta Method and Ordinary Runge-Kutta Method }
\label{tab:m-n-sqrt}
\end{table} 

Table \ref{tab:m-n-sqrt} shows that the relative error is 4 orders greater in magnitude in the case of  the RK4-method compared to the MRK4-method.  This is in well agreement to the error analysis presented in section \ref{sec:error}. On the other hand, we know that the basic  operations used in multiplicative calculus are multiplication, division, calculation of the exponential function and calculation of the logarithm function, whereas in the Newtonian case, we only have to consider multiplication, summation, and subtraction. 

Let all numbers be of size $n$-bit. The computational complexities for the following arithmetic operations are:

\begin{center}
\begin{tabular}{ll}
addition and subtraction & $O(n)$ \\
multiplication and division  & $O(n^2)$ \\
exponential and logarithm & $O(n^{5/2})$ 
\end{tabular}
\end{center}

So, evidently the number of operations must be significantly smaller for the MRK4 compared to RK4. In order to consider the MRK4 as a serious alternative to RK4, the performance of MRK4 has to be at least comparable. Performance means, higher accuracy, i.e. smaller errors, for the same computation time. Therefore, the relative error as a function of the  computation time has been measured by keeping the starting and end point fixed and varying the step size $h$. The results for both methods have been compared in figure \ref{fig3}.

\begin{figure}[H]

\centerline{\includegraphics[width=10cm]{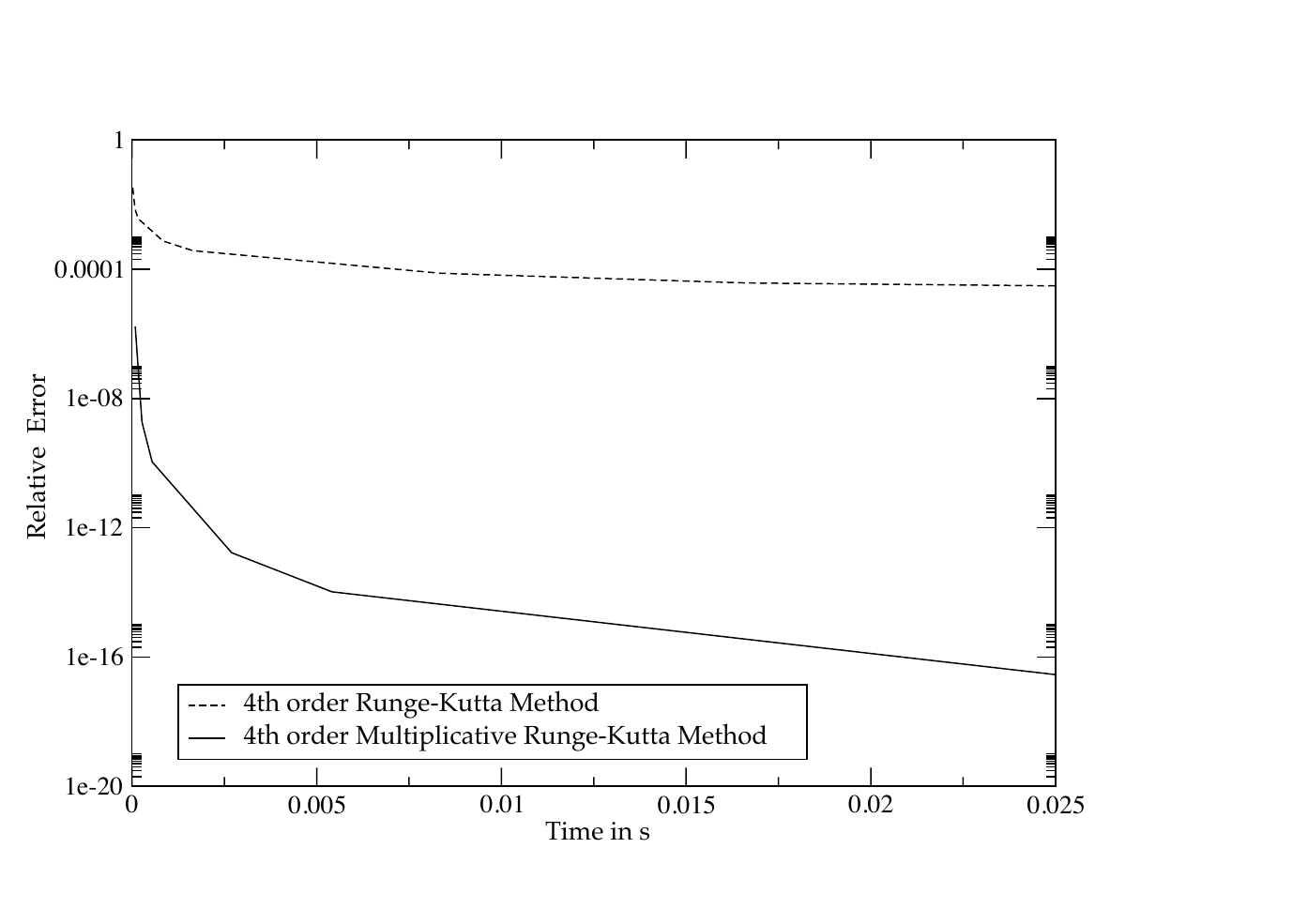}}
\caption{Comparison of the computation time and the relative error for the multiplicative initial value problem \eqref{eq:msqrt} and the initial value problem \eqref{eq:nsqrt} for the same initial value $x_0=0$ and $y_0=1$ and fixed final values $x_n=3$, $y_n=2$ by varying $h$.}
\label{fig3}
\end{figure}

The comparison of the relative errors as function of the computation time shows that the MRK4 method is working more efficiently compared to the RK4 method, showing a significant difference between the results. This comparison has been carried out also for other sample problems with known closed form solutions. The results show that the MRK4 method is more efficient compared to the RK4. 

\subsubsection{Biological Example}
\label{subsec:exmrk}	

In order to show that the proposed method is also applicable to get better results for mathematical models in biology,  we want to discuss the bacterial growth in food modelled by Huang \cite{huang08,huang10,huang12}.

In the Baranyi model \cite{baranyi94, baranyi95} for the bacterial growth in food described by the differential equation.

\begin{equation}
y'(t) = \mu_{max}\frac{1-e^{y-y_{max}} }{1+e^{-\alpha (t-\lambda)}} 
\label{eq:bakn}
\end{equation}

The multiplicative counter part of  equation \eqref{eq:bakn} is:

\begin{equation}
y^*(t) = \exp \left\{ \frac{\mu_{max}}{y}\,\frac{1-e^{y-y_{max}} }{1+e^{-\alpha (t-\lambda)}} \right\}
\label{eq:bakm}
\end{equation}

with the initial value $y_0=y(0)=7$.  As there is no closed form solution available for these initial value problems, we have solved both initial value problems using the MRK4 and the RK4 for small $h$, where both solutions coincide. Then we increased the step size $h$ and checked which method deviates first from the solutions for small $h$. As depicted in figure \ref{fig1}, the RK4 method deviates first from the accurate solution. We compared then the greatest $h$, where MRK4 still coincides with the solutions for small $h$, and the RK4 method for this $h$. Also in this case, the MRK4 method is giving better performance results compared to RK4. 

\begin{figure}[H]
\centerline{\includegraphics[width=8cm]{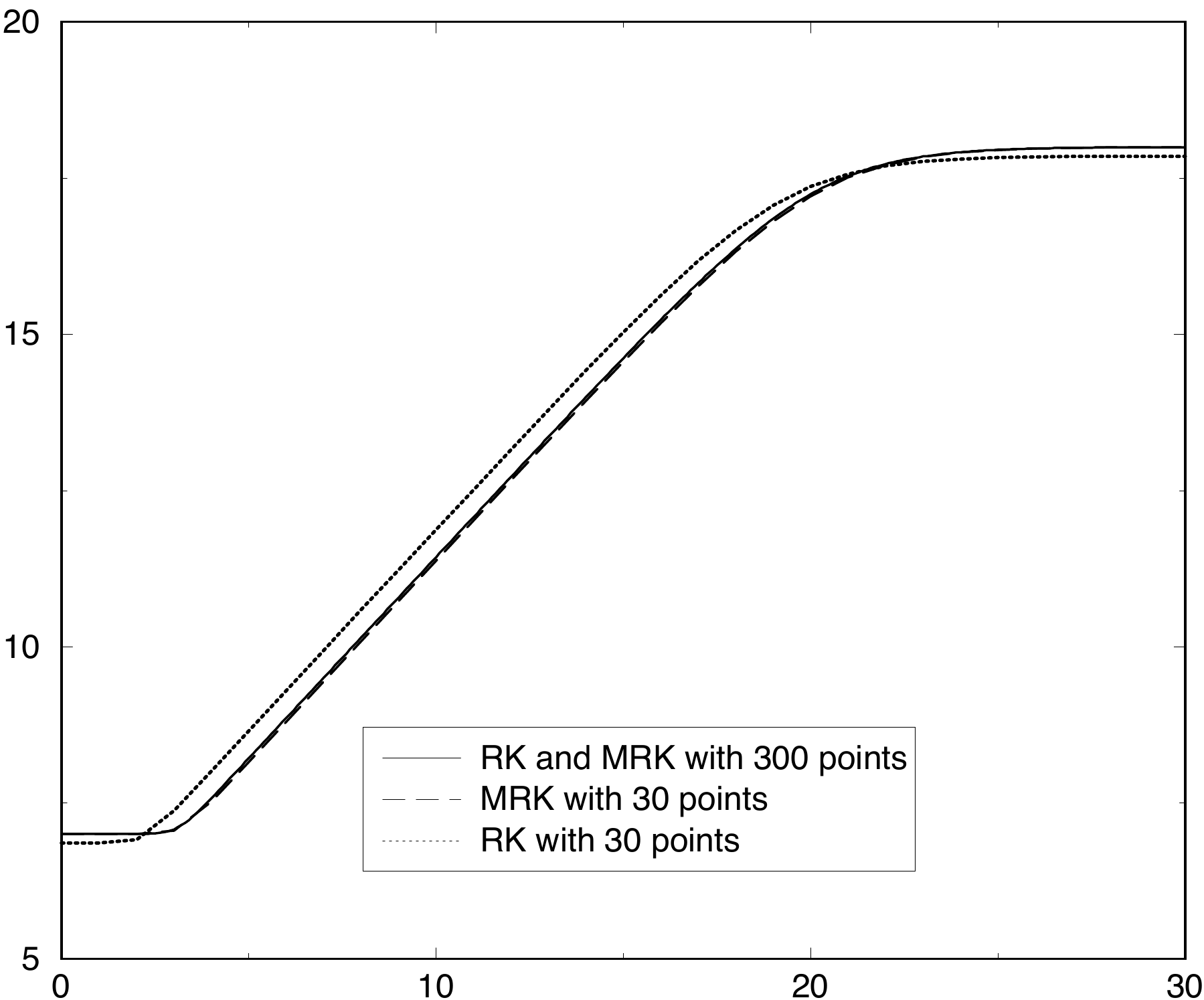}}
\caption{Solution for Bacteria growth model, $\lambda=3.21$, $\mu_{max}= 0.644$, $\alpha=4$, $y_{max}=18$.}
\label{fig1}
\end{figure}
The numerical solutions of the differential equations \eqref{eq:bakn} and \eqref{eq:bakm} using the corresponding Runge-Kutta Methods are not distinguishable for $h=0.1$. But, as depicted in figure \ref{fig1}, the MRK4 - method for $h=1$ still coincides with the solution for $h=0.1$, whereas the RK4 is significantly different (dotted line).

\subsection{Solution of a second order multiplicative differential equation}
\label{sec:cmrk}

As an example for higher order multiplicative initial value problems we will consider the following second order initial value problem
\begin{equation}
y^{**}(x)=f(x,y,y^*), \quad y(x_0) =y_0,\quad \text{ and }
y^*(x_0)= y_1.
\end{equation}

This initial value problem can be solved by solving the coupled system of first order multiplicative differential equations
\begin{eqnarray}
y_0^*(x) &=& y_1(x)\\
y_1^*(x)&=& f(x,y_0,y_1) . 
\end{eqnarray}

Exemplarily we want to solve the initial value problem for the 2nd order multiplicative differential equation
\begin{equation}
y^{**}(x) =e. 
\label{eq:f1}
\end{equation}
The corresponding ordinary second order differential equation is 
\begin{equation}
y''(x) = \frac{y'(x)^2}{y(x)}+ y(x).
\label{eq:f2} 
\end{equation}
The general solution of the differential equations \eqref{eq:f1} and  \eqref{eq:f2} is
\begin{equation}
y(x) =\alpha \exp \left\{\frac{x^2}{2}+\beta x\right\}. 
\end{equation}
This initial value problem, was solved also as an example for a multiplicative boundary value problem in \cite{ROK}. 
In order to be able to compare the results with the multiplicative finite difference methods solution, discussed in \cite{ROK}, we select $\alpha=1$,$\beta=1$, $x_0=1$, and $h=0.25$, resulting in the initial conditions
\begin{equation}
y_0=e^{3/2} \quad \text{ and } y_1=e^{2},
\label{eq:f1ic}
\end{equation}
and compare the results in the following table.

\begin{table}[H]
\[
\begin{array}{l l l l l l}
x &y_{exact}  &y_{MRK}   & \text{relative}& y_{MFD} & \text{relative} \\
  &                 &              & err_{MRK} \text{in \% }    &              &  err_{MFD} \text{in \% } \\
\hline
1 & 4.48168907 & 4.481689070  & 0 & 4.48168907 & 0 \\
1.25 & 7.62360992 & 7.62360992 & 9.3 \times 10^{-15} & 7.62360991 & 3.5 \times 10^{-13} \\
1.5 & 13.80457419 & 13.80457419 & 1.3\times 10^{-14} & 13.80457418 & 5.3 \times 10^{-13} \\ 
1.75 & 26.60901319 & 26.60901319  & 1.7\times 10^{-14} & 26.60913187 & 1.8 \times 10^{-13}\\
\hline 
\end{array}
\]

\caption{ comparison of the Multiplicative Runge-Kutta Method and Multiplicative Finite Difference Method }
\label{tab:mfd}
\end{table}

Table \ref{tab:mfd} shows  the numerical  approximation using the MRK4 for \eqref{eq:f1} with the initial conditions \eqref{eq:f1ic} and the corresponding results for the multiplicative finite difference method from  \cite{ROK}. In this case we can see that the MRK4 is slightly better than the Multiplicative Finite Difference method by one order of magnitude in the relative error.  On the other hand, if we solve the corresponding ordinary differential equation \eqref{eq:f2} with the corresponding initial values 

\begin{equation}
y_0=e^{3/2} \quad \text{ and } y_1=2 e^{3/2}
\label{eq:f2ic}
\end{equation}
we get the results as shown in table \ref{tab:mrk} below.

\begin{table}[H]
\[
\begin{array}{l l l l l l}
x &y_{exact}  &y_{MRK}   & \text{relative}& y_{newt} & \text{relative} \\
  &                 &              & err_{MRK} \text{in \% }    &              &  err_{newt} \text{in \% } \\
\hline
1 & 4.48168907 & 4.481689070  & 0 & 4.48168907 & 0 \\
1.25 & 7.62360992 & 7.62360992 & 9.3 \times 10^{-15} & 7.61823131 & 7.1 \times 10^{-2} \\
1.5 & 13.80457419 & 13.80457419 & 1.3\times 10^{-14} & 13.77941017 & 1.8 \times 10^{-1} \\ 
1.75 & 26.60901319 & 26.60901319  & 1.7\times 10^{-14} & 26.51619718 & 3.5 \times 10^{-1} \\
\hline
\end{array}
\]

\caption{ comparison of the Multiplicative Runge-Kutta Method and the Runge-Kutta Method }
\label{tab:mrk}
\end{table} 

Obviously,  the RK4 fails drastically in this case, as the relative error differs by 13 orders in magnitude compared to its multiplicative counterpart.  The MRK4, as well as the multiplicative finite difference method succeed to produce proper results. Also in this case the performance of the MRK4 method is significantly better compared to the RK4.

\section{Conclusion}
\label{sec:con}

After a short motivation of the problem in the introduction, we described the multiplicative Runge-Kutta method for the solution of multiplicative initial value problems  of the form 
\[
y^*(x)= f(x,y), \text{ with } y(x_0)=y_0,
\]
where $x_0$ is the starting point and $y_0$ the initial value. The derivation of the 2nd order multiplicative Runge-Kutta method was carried out explicitly in detail. For the higher order methods the ansatzes, the solutions, as well as the corresponding Butcher tableaus are presented. The most successful methods to overcome the restrictions of Multiplicative Calculus are presented in section \ref{sec:extension}. These methods ensure that the Multiplicative Runge-Kutta method becomes a universally applicable tool. The error analysis and the convergence of multiplicative one-step methods was discussed in detail in section \ref{sec:error}. Finally the Multiplicative Runge-Kutta method is applied to several problems, and the results are compared with the results from the ordinary Runge-Kutta method and the Multiplicative Finite Difference Method. We could observe, that in these examples the Multiplicative Runge-Kutta method produces significantly better results for the same step width compared to the ordinary Runge-Kutta method. Furthermore, the performance of both methods was compared for one example explicitly. We observed that the Multiplicative Runge-Kutta method produced smaller errors for the same computation time compared to the ordinary Runge-Kutta method, demonstrating the universal applicability of the proposed method. The Multiplicative Runge-Kutta method was also applied to the solution of a bacterial growth model proposed by Baranyi and compared to the ordinary Runge-Kutta method, resembling the previous results.

\bibliography{multiplicative_calculus}
\bibliographystyle{plain}

\end{document}